\documentclass{article} 
\usepackage{amsmath}
\usepackage{amsthm}
\usepackage{amssymb}
\usepackage{amsfonts}
\newtheorem{theorem}{Theorem}

\newtheorem{corollary}[theorem]{Corollary}

\newtheorem{proposition}[theorem]{Proposition}
\theoremstyle{definition}
\newtheorem{remark}[theorem]{Remark}

\renewenvironment{proof}[1][Proof]{\noindent\textbf{#1.} }{\hfill$\Box$}

\def\tprod{\mathop{\prod}}
\def\zx{\begin{equation}\label}
\def\zc{\end{equation}}
\def\aru#1{\left\{{\begin{array}{l}#1\end{array}}\right.}
\def\mm{\medskip\\}
\def\dsum{\displaystyle\sum}
\def\dfrac{\displaystyle\frac}
\def\vcov#1{\!\!\!\left.\phantom{\int}\right|{}^{(1)}_{(#1)}}
\def\svcov#1{\!\!\!\left.\phantom{\sum}\right|{}^{(1)}_{(#1)}}
\numberwithin{theorem}{section}
\numberwithin{equation}{section}
\begin{document}
\title{The $1$-jet Generalized Lagrange Geometry induced by the rheonomic Chernov metric}
\author{Vladimir Balan and Mircea Neagu}
\date{}
\maketitle
\begin{abstract}
{The aim of this paper is to develop on the $1$-jet space $J^{1}(\mathbb{R},M^{4})$ the jet Generalized Lagrange Geometry (cf. \cite{Ner,Nec}) for the rheonomic Chernov metric}
    $$F_{[3]}(t,y)=\sqrt{h^{11}(t)}\cdot \sqrt[3]{y_{1}^{1}y_{1}^{2}y_{1}^{3}+
        y_{1}^{1}y_{1}^{2}y_{1}^{4}+y_{1}^{1}y_{1}^{3}y_{1}^{4}+y_{1}^{2}y_{1}^{3}y_{1}^{4}}.$$
{The associated gravitational and electromagnetic field models based on the rheonomic Finsler Chernov metric tensor are developed and discussed.}
\end{abstract}



\section{Introduction}

\hspace{5mm}It is obvious that our natural physical intuition distinguishes
four dimensions in a natural correspondence with {material environment}.
Consequently, {four-dimensionality plays a special} role in almost all
modern physical theories.

On the other hand, it is an well known fact that, in order to create the
Relativity Theory, Einstein was forced to use the Riemannian geometry
instead of the classical Euclidean geometry, the Riemannian geometry
representing the natural ma\-the\-ma\-ti\-cal model for the local
\textit{isotropic} space-time. {But recent studies of physicists suggest
a \textit{non-isotropic} perspective of Space-Time - e.g.,  the concept of inertial body mass
emphasizes the necessity of study of locally non-isotropic spaces
(\cite{GP}).}
{Among the possible models for the study of non-isotropic physical phenomena,
Finsler geometry is an appropriate and effective ma\-the\-ma\-ti\-cal framework.}

The {works of Russian scholars (\cite{As1,Mikhailov,GP})} emphasize the
importance of the Finsler geometry which is characterized by the {complete
equivalence} of all non-isotropic directions and {promote in their works
models based on special locally-Minkowski types of $m-$root metrics - e.g.,
Berwald-Mo\'{o}r or Chernov. Since any of the directions can be related to proper time,
such spaces were generically called as having \textit{"multi-dimensional time"} (\cite{Pa2}).
In the framework of the $3$- and $4$-dimensional linear space with Berwald-Mo\'{o}r
metric (i.e., having three- and four-dimensional time), Pavlov and his
co-workers (\cite{GP,Pa1,Pa2}) provide new physical model-supporting evidence
and geometrical interpretations,} such as:
\begin{itemize}\setlength{\parskip}{-1mm}
\item physical events = points in the multi-dimensional time;
\item straight lines = shortest curves;
\item intervals = distances between the points along of a straight line;
\item light pyramids $\Leftrightarrow $ light cones in a pseudo-Euclidian space;
\item {surfaces of simultaneity} = the surfaces of simultaneous physical events.
\end{itemize}
An important model of $m-$root type - the Chernov metric (\cite{Che,Ba}),
   \zx{ch}F:TM\rightarrow \mathbb{R},\mathbb{\qquad }F(y)=
    \sqrt[3]{y^{1}y^{2}y^{3}+y^{1}y^{2}y^{4}+y^{1}y^{3}y^{4}+y^{2}y^{3}y^{4}},\zc
was recently shown to be relevant for Relativity. The larger class of Finsler metrics to which this metric belongs, the $m-$root metrics, have been previously studied by the Japanese geometers Matsumoto and
Shimada (\cite{Matsumoto,Mats-Shimada,Shimada}).

Considering the former geometrical and physical reasons, the present paper is devoted to the development on the $1$-jet space $J^{1}(\mathbb{R},M^{4})$ of the Finsler-like geometry, applied to geometric gravitational and electromagnetic field theory associated to the natural 1-time rheonomic jet extension of the Chernov \linebreak
metric \eqref{ch}
   \zx{chr}F_{[3]}(t,y)=\sqrt{h^{11}(t)}\cdot \sqrt[3]{y_{1}^{1}y_{1}^{2}y_{1}^{3}+
    y_{1}^{1}y_{1}^{2}y_{1}^{4}+y_{1}^{1}y_{1}^{3}y_{1}^{4}+y_{1}^{2}y_{1}^{3}y_{1}^{4}},\zc
where $h_{11}(t)$ is a Riemannian metric on $\mathbb{R}$ and $(t,x^{1},x^{2},x^{3},x^{4},
y_{1}^{1},y_{1}^{2},y_{1}^{3},y_{1}^{4})$ are the coordinates of the $1$-jet space $J^{1}(\mathbb{R},M^{4})$.

{The geometry that models gravitational and electromagnetic theories, relying on distinguished ($d-$)connections (and their $d-$torsions and $d-$curvatures), produced by a jet rheonomic Lagrangian function $L:J^{1}(\mathbb{R}%
,M^{n})\rightarrow \mathbb{R}$, was extensively described in \cite{Ner}, where the geometrical ideas are similar,
but exhibiting distinct features compared to the ones developed by Miron and Anastasiei in classical Generalized Lagrange Geometry (\cite{MA}). The geometrical jet distinguished framework from \cite{Ner} - generically called as jet geometrical theory of the \textit{rheonomic Lagrange spaces}, was initially stated by Asanov in \cite{As2} and developed further in the book \cite{Nec}.}

In the sequel, we apply the general geometrical results from \cite{Ner} to the rheonomic Chernov metric $F_{[3]}$.

\section{Preliminary notations and formulas}

\hspace{5mm}Let $(\mathbb{R},h_{11}(t))$ be a Riemannian manifold, where $\mathbb{R}$ is the set of real numbers. The Christoffel symbol of the Riemannian metric $h_{11}(t)$ is
   \zx{na}\varkappa _{11}^{1}=\frac{h^{11}}{2}\frac{dh_{11}}{dt},\quad\mbox{ {where} }\quad
    h^{11}=\frac{1}{h_{11}}>0.\zc
Let also $M^{4}$ be a manifold of dimension four, whose local coordinates
are $(x^{1},x^{2},x^{3},x^{4})$. Let us consider the $1$-jet space $J^{1}(%
\mathbb{R},M^{4})$, whose local coordinates are%
   $$(t,x^{1},x^{2},x^{3},x^{4},y_{1}^{1},y_{1}^{2},y_{1}^{3},y_{1}^{4}).$$
These transform by the rules (the Einstein convention of summation is used
throughout this work):
   \zx{tr-rules}\widetilde{t}=\widetilde{t}(t),\quad \widetilde{x}^{p}=\widetilde{x}^{p}(x^{q}),\quad
   \widetilde{y}_{1}^{p}=\dfrac{\partial \widetilde{x}^{p}}{\partial x^{q}}\dfrac{dt}{d\widetilde{t}}\cdot
   y_{1}^{q},\qquad p,q=\overline{1,4},\zc
where $d\widetilde{t}/dt\neq 0$ and rank $(\partial \widetilde{x}^{p}/\partial x^{q})=4$. \par\smallskip
{We further consider that the manifold $M^{4}$ is endowed with a tensor of kind $(0,3)$, given by the local components $S_{pqr}(x)$, which is totally symmetric in the indices $p$, $q$ and $r$. We shall use the notations}
   \zx{nota}S_{ij1}=6S_{ijp}y_{1}^{p},\quad S_{i11}=3S_{ipq}(x)y_{1}^{p}y_{1}^{q},\quad
    S_{111}=S_{pqr}y_{1}^{p}y_{1}^{q}y_{1}^{r}\zc
{We assume that the $d-$tensor $S_{ij1}$ is non-degenerate, i.e.,} there exists the $d-$tensor $S^{jk1}$ on $J^{1}(\mathbb{R},M^{4})$, such that $S_{ij1}S^{jk1}=\delta _{i}^{k}.$ In this context, we can consider the \textit{third-root Finsler-like function} {(\cite{Shimada}, \cite{AN}), which is $1$-positive homogenous in the variable $y$},
   \zx{F}F(t,x,y)=\sqrt[3]{S_{pqr}(x)y_{1}^{p}y_{1}^{q}y_{1}^{r}}\cdot \sqrt{h^{11}(t)}=
    \sqrt[3]{S_{111}(x,y)}\cdot \sqrt{h^{11}(t)},\zc
where the Finsler function $F$ has as domain of definition all values $(t,x,y)$ which {satisfy} the condition $S_{111}(x,y)\neq 0.$ {Then the} $3$-positive homogeneity of the "$y$-function" $S_{111}$ ({which is a} $d-$tensor on the $1$-jet space $J^{1}(\mathbb{R},M^{4})$), leads to the equalities:%
   $$S_{i11}=\frac{\partial S_{111}}{\partial y_{1}^{i}},\quad S_{i11}y_{1}^{i}=3S_{111},\quad
        S_{ij1}y_{1}^{j}=2S_{i11},\quad S_{ij1}=\frac{\partial S_{i11}}{\partial y_{1}^{j}}=
       \frac{\partial ^{2}S_{111}}{\partial y_{1}^{i}\partial y_{1}^{j}},$$
   $$S_{ij1}y_{1}^{i}y_{1}^{j}=6S_{111},\quad \frac{\partial S_{ij1}}{\partial y_{1}^{k}}=6S_{ijk},\quad
        S_{ijp}y_{1}^{p}=\frac{1}{6}S_{ij1}.$$
The \textit{fundamental metrical $d-$tensor} produced by $F$ is given by the
formula%
   $$g_{ij}(t,x,y)=\frac{h_{11}(t)}{2}\frac{\partial ^{2}F^{2}}{\partial y_{1}^{i}\partial y_{1}^{j}}.$$
By direct computations, the fundamental metrical $d-$tensor takes the form%
\begin{equation}
g_{ij}(x,y)=\frac{S_{111}^{-1/3}}{3}\left[ S_{ij1}-\frac{1}{3S_{111}}%
S_{i11}S_{j11}\right] .  \label{g-(ij)-general}
\end{equation}
{Moreover, since the $d-$tensor $S_{ij1}$ is non-degenerate, the matrix $g=(g_{ij})$ admits an inverse $g^{-1}=(g^{jk})$, whose entries are}
\begin{equation}g^{jk}=3S_{111}^{1/3}\left[ S^{jk1}+\frac{S_{1}^{j}S_{1}^{k}}{3\left(
S_{111}-\mathbf{S}_{111}\right) }\right] ,  \label{g+(jk)-general}
\end{equation}
where $S_{1}^{j}=S^{jp1}S_{p11}$ and $3\mathbf{S}_{111}=S^{pq1}S_{p11}S_{q11}.$\par\smallskip

Following {the ideas from \cite{Ner}}, the \textit{energy action functional}
   $$\mathbb{E}(t,x(t))=\int_{a}^{b}F^{2}(t,x(t),y(t))\sqrt{h_{11}(t)}
    dt=\int_{a}^{b}S_{111}^{2/3}\cdot h^{11}\sqrt{h_{11}}dt,$$
where $y(t)=dx/dt$, produces on the $1$-jet space $J^{1}(\mathbb{R},M^{4})$,
via the Euler-Lagrange equations, the \textit{canonical time dependent spray}
   \zx{sp}\mathcal{S}=\left( H_{(1)1}^{(i)},\mbox{ }G_{(1)1}^{(i)}\right) ,\zc
where, {using the notations \eqref{na} and \eqref{nota}, we have}
   $$H_{(1)1}^{(i)}=-\frac{\varkappa _{11}^{1}}{2}y_{1}^{i}$$
and
{
\begin{eqnarray}
G_{(1)1}^{(i)} &=&\frac{g^{im}}{6\sqrt[3]{S_{111}}}\left[ \frac{\partial
S_{m11}}{\partial x^{p}}y_{1}^{p}-\left( 1-\varkappa _{11}^{1}\right) \frac{%
\partial S_{111}}{\partial x^{m}}\right] -  \label{spray-spatial-general} \\
&&-\frac{S_{1}^{i}}{6\left( S_{111}-\mathbf{S}_{111}\right) }\left( \frac{\partial
S_{111}}{\partial x^{p}}y_{1}^{p}+3\varkappa _{11}^{1}S_{111}\right).  \notag
\end{eqnarray}
}
\begin{remark}
In the particular case when the components $G_{pqr}$ are independent on the
variable $x$, the expression of (\ref{spray-spatial-general}) simplifies as%
\begin{equation}
G_{(1)1}^{(i)}=-\varkappa _{11}^{1}\frac{S_{111}}{2\left( S_{111}-\mathbf{S}%
_{111}\right) }S_{1}^{i}.  \label{spray-spatial-Minkowski}
\end{equation}%
Note that, in this case, the Finsler-like function (\ref{F}) {is of locally-Minkowski type}.
\end{remark}

It is known (\cite{Ner}) that the canonical time dependent spray $\mathcal{S}$ given by \eqref{sp} determines on the $1$-jet space $J^{1}(\mathbb{R},M^{4})$ a \textit{canonical nonlinear connection} given by
   \zx{nlc-general}\Gamma =\left( M_{(1)1}^{(i)}=2H_{(1)1}^{(i)}=-\varkappa _{11}^{1}y_{1}^{i},%
    \mbox{ }N_{(1)j}^{(i)}=\frac{G_{(1)1}^{(i)}}{\partial y_{1}^{j}}\right) .\zc
\section{The rheonomic Chernov metric}

\hspace{5mm}Beginning with this Section we will focus only on the \textit{%
rheonomic Chernov metric}, which is the Finsler-like metric (\ref{F}) for
the particular case%
\begin{equation*}
S_{pqr}:=S_{[3]pqr}=\left\{
\begin{array}{ll}
\dfrac{1}{3!}, & \{p,q,r\}\mbox{ - distinct indices}\medskip \\
0, & \mbox{otherwise.}%
\end{array}%
\right.
\end{equation*}%
Consequently, the rheonomic Chernov metric is given by%
\begin{equation}
F_{[3]}(t,y)=\sqrt{h^{11}(t)}\cdot \sqrt[3]{%
y_{1}^{1}y_{1}^{2}y_{1}^{3}+y_{1}^{1}y_{1}^{2}y_{1}^{4}+y_{1}^{1}y_{1}^{3}y_{1}^{4}+y_{1}^{2}y_{1}^{3}y_{1}^{4}%
}.  \label{rheon-B-M}
\end{equation}%
Moreover, using {the} preceding notations and formulas, we obtain the following
relations:%
\begin{equation*}
S_{111}:=%
S_{[3]111}=y_{1}^{1}y_{1}^{2}y_{1}^{3}+y_{1}^{1}y_{1}^{2}y_{1}^{4}+y_{1}^{1}y_{1}^{3}y_{1}^{4}+y_{1}^{2}y_{1}^{3}y_{1}^{4},
\end{equation*}%
\begin{equation*}
S_{i11}:=S_{[3]i11}=\frac{\partial S_{[3]111}}{\partial
y_{1}^{i}}=\frac{S_{[3]111}y_{1}^{i}-S_{[4]1111}}{\left( y_{1}^{i}\right)
^{2}},
\end{equation*}%
\begin{equation*}
S_{ij1}:=S_{[3]ij1}=\frac{\partial S_{[3]i11}}{\partial
y_{1}^{j}}=\frac{\partial ^{2}S_{[3]111}}{\partial y_{1}^{i}\partial
y_{1}^{j}}=\left\{
\begin{array}{ll}
S_{[1]1}-y_{1}^{i}-y_{1}^{j}, & i\neq j\medskip \\
0, & i=j,%
\end{array}%
\right.
\end{equation*}%
where $S_{[4]1111}=y_{1}^{1}y_{1}^{2}y_{1}^{3}y_{1}^{4}$ and $%
S_{[1]1}=y_{1}^{1}+y_{1}^{2}+y_{1}^{3}+y_{1}^{4}.$ Note that, for $i\neq j$,
the following equality {holds true as well}:
\begin{equation*}
S_{[3]i11}\cdot S_{[3]j11}=S_{[3]111}\left(
S_{[1]1}-y_{1}^{i}-y_{1}^{j}\right) +\frac{S_{[4]1111}^{2}}{\left(
y_{1}^{i}\right) ^{2}\left( y_{1}^{j}\right) ^{2}}.
\end{equation*}

Because we have
   $$0\neq \det \left( S_{ij1}\right) _{i,j=\overline{1,4}}=4\left[4S_{[4]1111}-S_{[1]1}S_{[3]111}\right]
        :=\mathbf{D}_{1111},$$
we find
\begin{equation*}
S^{jk1}:=S_{[3]}^{jk1}=\left\{
\begin{array}{ll}
\dfrac{-2}{\mathbf{D}_{1111}}\left( y_{1}^{j}+y_{1}^{k}\right) \left[
y_{1}^{j}y_{1}^{k}+\dfrac{S_{[4]1111}}{y_{1}^{j}y_{1}^{k}}\right] , & j\neq
k\medskip \\
\dfrac{1}{\mathbf{D}_{1111}}\cdot \dfrac{1}{y_{1}^{j}}\left[
\tprod_{l=1}^{4}\left( y_{1}^{j}+y_{1}^{l}\right) \right] , & j=k.%
\end{array}%
\right.
\end{equation*}%
{Further, laborious computations lead to:}
\begin{equation}
\begin{array}{l}
S_{1}^{j}:=S_{[3]1}^{j}=S_{[3]}^{jp1}S_{[3]p11}=\dfrac{1}{2}%
y_{1}^{j},\medskip \\
\mathbf{S}_{111}:=\mathbf{S}%
_{[3]111}=S_{[3]}^{pq1}S_{[3]p11}S_{[3]q11}=\dfrac{1}{2}S_{[3]111}.%
\end{array}
\label{formule}
\end{equation}

Replacing now the above computed entities into the formulas (\ref%
{g-(ij)-general}) and (\ref{g+(jk)-general}), we get $g_{ij}:=%
g_{[3]ij}=$
\begin{equation}
=\left\{
\begin{array}{ll}
\dfrac{S_{[3]111}^{-1/3}}{9}\left[ 2\left(
S_{[1]1}-y_{1}^{i}-y_{1}^{j}\right) -\dfrac{S_{[4]1111}^{2}}{S_{[3]111}}%
\cdot \dfrac{1}{\left( y_{1}^{i}\right) ^{2}\left( y_{1}^{j}\right) ^{2}}%
\right] , & i\neq j\medskip \\
\dfrac{-S_{[3]111}^{-4/3}}{9}\cdot S_{[3]i11}^{2}, & i=j%
\end{array}%
\right.  \label{g-jos-(ij)}
\end{equation}

and%
\begin{equation}
g^{jk}:=g_{[3]}^{jk}=3S_{[3]111}^{1/3}\left[ S^{jk1}+\frac{1}{%
6S_{[3]111}}y_{1}^{j}y_{1}^{k}\right] \mbox{.}  \label{g-sus-(jk)}
\end{equation}

Consequently, using the formulas (\ref{spray-spatial-Minkowski}) and (\ref%
{formule}), we find the following geometrical result:

\begin{proposition}
For the rheonomic Chernov metric (\ref{rheon-B-M}), the \textit{energy
action functional}%
\begin{equation*}
\mathbb{E}_{[3]}(t,x(t))=\int_{a}^{b}\sqrt[3]{\left(
y_{1}^{1}y_{1}^{2}y_{1}^{3}+y_{1}^{1}y_{1}^{2}y_{1}^{4}+y_{1}^{1}y_{1}^{3}y_{1}^{4}+y_{1}^{2}y_{1}^{3}y_{1}^{4}\right) ^{2}%
}\cdot h^{11}\sqrt{h_{11}}dt
\end{equation*}%
produces on the $1$-jet space $J^{1}(\mathbb{R},M^{4})$ the \textit{%
canonical time dependent }spray
\begin{equation}
\mathcal{S}_{[3]}=\left( H_{(1)1}^{(i)}=-\frac{\varkappa _{11}^{1}}{2}%
y_{1}^{i},\mbox{ }G_{(1)1}^{(i)}=-\frac{\varkappa _{11}^{1}}{2}%
y_{1}^{i}\right) .  \label{spray-B-M}
\end{equation}
\end{proposition}

Moreover, the formulas (\ref{nlc-general}) and (\ref{spray-B-M}) imply

\begin{corollary}
The canonical nonlinear connection on the $1$-jet space $J^{1}(\mathbb{R}%
,M^{4})$ produced by the rheonomic Chernov metric (\ref{rheon-B-M}) is%
\begin{equation}
\Gamma _{\lbrack 3]}=\left( M_{(1)1}^{(i)}=-\varkappa _{11}^{1}y_{1}^{i},%
\mbox{ }N_{(1)j}^{(i)}=-\frac{\varkappa _{11}^{1}}{2}\delta _{j}^{i}\right) ,
\label{nlc-B-M}
\end{equation}%
where $\delta _{j}^{i}$ is the Kronecker symbol.
\end{corollary}

\section{Cartan canonical connection.\\Distinguished torsions and curvatures}

\hspace{5mm}The importance of the nonlinear connection (\ref{nlc-B-M}) is
coming from the possibility of construction of the dual \textit{local adapted bases}: of $d-$vector fields
    \zx{a-b-v}\left\{ \frac{\delta }{\delta t}=\frac{\partial }{\partial t}+\varkappa_{11}^{1}y_{1}^{p}
    \frac{\partial }{\partial y_{1}^{p}}\;\;;\;\;
    \frac{\delta}{\delta x^{i}}=\frac{\partial }{\partial x^{i}}+\frac{\varkappa _{11}^{1}}{2}
        \frac{\partial }{\partial y_{1}^{i}}\;\;;\;\;
        \dfrac{\partial }{\partial y_{1}^{i}}\right\} \subset \mathcal{X}(E)\zc
and of $d-$covector fields
   \zx{a-b-co}\left\{ dt\;\;;\;\;dx^{i}\;\;;\;\;
    \delta y_{1}^{i}=dy_{1}^{i}-\varkappa_{11}^{1}y_{1}^{i}dt-\frac{\varkappa _{11}^{1}}{2}dx^{i}\right\}
    \subset \mathcal{X}^{\ast }(E),\zc
where $E=J^{1}(\mathbb{R},M^{4})$. Note that, under a change of coordinates (%
\ref{tr-rules}), the elements of the adapted bases (\ref{a-b-v}) and (\ref%
{a-b-co}) transform as classical tensors. Consequently, all subsequent
geometrical objects on the $1$-jet space $J^{1}(\mathbb{R},M^{4})$ (as
Cartan canonical connection, torsion, curvature etc.) will be described in
local adapted components. We emphasize that the definition of local components of connections, torsion and curvature, obey the formalism used in the works \cite{MA,Nec,Ner,AN}.

Using a general result from \cite{Ner}, by direct computations, we
can give the following important geometrical result:

\begin{theorem}
The Cartan canonical $\Gamma _{\lbrack 3]}$-linear connection, produced by
the rheonomic Chernov metric (\ref{rheon-B-M}), has the following adapted
local components:%
\begin{equation*}
C\Gamma _{\lbrack 3]}=\left( \varkappa _{11}^{1},\mbox{ }G_{j1}^{k}=0,\mbox{
}L_{jk}^{i}=\frac{\varkappa _{11}^{1}}{2}C_{j(k)}^{i(1)},\mbox{ }%
C_{j(k)}^{i(1)}\right) ,
\end{equation*}%
where%
\begin{equation}
\begin{array}{ll}
C_{j(k)}^{i(1)}= & 3S_{[3]}^{im1}S_{[3]jkm}-\medskip \\
& -\dfrac{1}{6}\dfrac{1}{S_{[3]111}}\left[ S_{[3]jk1}\dfrac{y_{1}^{i}}{2}%
+\delta _{j}^{i}S_{[3]k11}+\delta _{k}^{i}S_{[3]j11}\right] +\medskip \\
& +\dfrac{1}{9}\dfrac{1}{S_{[3]111}^{2}}S_{[3]j11}S_{[3]k11}y_{1}^{i}\mbox{.}%
\end{array}
\label{C-(ijk)}
\end{equation}
\end{theorem}

\begin{proof}
Using the Chernov derivative operators (\ref{a-b-v}) and (\ref{a-b-co}),
together with the relations (\ref{g-jos-(ij)}) and (\ref{g-sus-(jk)}), we
apply the general formulas which give the adapted components of the Cartan
canonical connection, namely \cite{Ner}%
\begin{equation*}
G_{j1}^{k}=\frac{g_{[3]}^{km}}{2}\frac{\delta g_{[3]mj}}{\delta t},\quad
L_{jk}^{i}=\frac{g_{[3]}^{im}}{2}\left( \frac{\delta g_{[3]jm}}{\delta x^{k}}%
+\frac{\delta g_{[3]km}}{\delta x^{j}}-\frac{\delta g_{[3]jk}}{\delta x^{m}}%
\right) ,
\end{equation*}%
\begin{equation*}
C_{j(k)}^{i(1)}=\frac{g_{[3]}^{im}}{2}\left( \frac{\partial g_{[3]jm}}{%
\partial y_{1}^{k}}+\frac{\partial g_{[3]km}}{\partial y_{1}^{j}}-\frac{%
\partial g_{[3]jk}}{\partial y_{1}^{m}}\right) =\frac{g_{[3]}^{im}}{2}\frac{%
\partial g_{[3]jk}}{\partial y_{1}^{m}},
\end{equation*}%
where, by computations, we have%
\begin{eqnarray*}
\frac{\partial g_{[3]jk}}{\partial y_{1}^{m}}
&=&2S_{[3]111}^{-1/3}S_{[3]jkm}-\medskip \\
&&-\dfrac{1}{9}S_{[3]111}^{-4/3}\left\{
S_{[3]jk1}S_{[3]m11}+S_{[3]km1}S_{[3]j11}+S_{[3]mj1}S_{[3]k11}\right\}
+\medskip \\
&&+\frac{4}{27}S_{[3]111}^{-7/3}S_{[3]j11}S_{[3]k11}S_{[3]m11}.
\end{eqnarray*}%
{For details, we refer to \cite{Shimada} and \cite{AN}.}
\end{proof}

\begin{remark}
The {following} properties of the $d-$tensor $C_{j(k)}^{i(1)}$ {hold} true:%
\begin{equation*}
C_{j(k)}^{i(1)}=C_{k(j)}^{i(1)},\quad C_{j(m)}^{i(1)}y_{1}^{m}=0.
\end{equation*}
\end{remark}

\begin{theorem}
The Cartan canonical connection $C\Gamma _{\lbrack 3]}$ of the rheonomic\linebreak
Chernov metric (\ref{rheon-B-M}) has \textbf{three} effective local torsion
d-tensors:
{
   $$\begin{array}{c}P_{(1)i(j)}^{(k)\mbox{ }(1)}=-\dfrac{1}{2}\varkappa_{11}^{1}C_{i(j)}^{k(1)},\quad
        P_{i(j)}^{k(1)}=C_{i(j)}^{k(1)},\medskip \\
    R_{(1)1j}^{(k)}=\dfrac{1}{2}\left( \dfrac{d\varkappa _{11}^{1}}{dt}-\varkappa _{11}^{1}
        \varkappa _{11}^{1}\right) \delta _{j}^{k}.\end{array}$$
}
\end{theorem}

\begin{proof}
A general $h$-normal $\Gamma $-linear connection on the 1-jet space $J^{1}(%
\mathbb{R},M^{4})$ is characterized by \textit{eight} effective $d-$tensors of
torsion {(cf. \cite{Ner})}. For our Cartan
canonical connection $C\Gamma _{\lbrack 3]}$ these reduce to the following
\textit{three} (the other five cancel):\par\smallskip
   ${P_{(1)i(j)}^{(k)\mbox{ }(1)}={\dfrac{\partial N_{(1)i}^{(k)}}{\partial
y_{1}^{j}}}-L_{ji}^{k}},\quad R_{(1)1j}^{(k)}={\dfrac{\delta M_{(1)1}^{(k)}}{%
\delta x^{j}}}-{\dfrac{\delta N_{(1)j}^{(k)}}{\delta t}},\quad P_{i(j)}^{k(1)}=C_{i(j)}^{k(1)}.$
\end{proof}

\begin{theorem}
The Cartan canonical connection $C\Gamma _{\lbrack 3]}$ of the rheonomic\linebreak
Chernov metric (\ref{rheon-B-M}) has \textbf{three} effective local
curvature $d-$tensors:%
\begin{equation*}
\begin{array}{c}
R_{ijk}^{l}=\dfrac{1}{4}\varkappa _{11}^{1}\varkappa
_{11}^{1}S_{i(j)(k)}^{l(1)(1)},\quad P_{ij(k)}^{l\mbox{ }(1)}=\dfrac{1}{2}%
\varkappa _{11}^{1}S_{i(j)(k)}^{l(1)(1)},\medskip \\
S_{i(j)(k)}^{l(1)(1)}={{\dfrac{\partial C_{i(j)}^{l(1)}}{\partial y_{1}^{k}}}%
-{\dfrac{\partial C_{i(k)}^{l(1)}}{\partial y_{1}^{j}}}%
+C_{i(j)}^{m(1)}C_{m(k)}^{l(1)}-C_{i(k)}^{m(1)}C_{m(j)}^{l(1)}.}%
\end{array}%
\end{equation*}
\end{theorem}

\begin{proof}
A general $h$-normal $\Gamma $-linear connection on the 1-jet space $J^{1}(%
\mathbb{R},M^{4})$ is characterized by \textit{five} effective $d-$tensors of
curvature {(cf. \cite{Ner})}. For our Cartan
canonical connection $C\Gamma _{\lbrack 3]}$ these reduce to the following
\textit{three} (the other two cancel):%
\begin{equation*}
\begin{array}{l}
\medskip {R_{ijk}^{l}={\dfrac{\delta L_{ij}^{l}}{\delta x^{k}}}-{\dfrac{%
\delta L_{ik}^{l}}{\delta x^{j}}}+L_{ij}^{m}L_{mk}^{l}-L_{ik}^{m}L_{mj}^{l},}
\\
\medskip {P_{ij(k)}^{l\;\;(1)}={\dfrac{\partial L_{ij}^{l}}{\partial
y_{1}^{k}}}-C_{i(k)|j}^{l(1)}+C_{i(m)}^{l(1)}P_{(1)j(k)}^{(m)\;\;(1)},} \\
S_{i(j)(k)}^{l(1)(1)}={{\dfrac{\partial C_{i(j)}^{l(1)}}{\partial y_{1}^{k}}}%
-{\dfrac{\partial C_{i(k)}^{l(1)}}{\partial y_{1}^{j}}}%
+C_{i(j)}^{m(1)}C_{m(k)}^{l(1)}-C_{i(k)}^{m(1)}C_{m(j)}^{l(1)},}%
\end{array}%
\end{equation*}%
where%
\begin{equation*}
{C_{i(k)|j}^{l(1)}=}\frac{\delta {C_{i(k)}^{l(1)}}}{\delta x^{j}}+{%
C_{i(k)}^{m(1)}L_{mj}^{l}}-{C_{m(k)}^{l(1)}L_{ij}^{m}}-{%
C_{i(m)}^{l(1)}L_{kj}^{m}}.
\end{equation*}
\end{proof}
\begin{remark}We have denoted by ${}_{/1}$, ${}_{|i}$ \ and $\vcov{i}$ \ the Cartan covariant 
derivatives with respect to the corresponding $\mathbb{R}-$horizontal (temporal), $M-$horizontal and vertical vector fields of the basis \eqref{a-b-v}.
\end{remark}
\section{\large Applications of the rheonomic Chernov metric}
\subsection{Geometrical gravitational theory}

\hspace{5mm}From a physical point of view, on the 1-jet space $J^{1}(\mathbb{%
R},M^{4})$, the rheonomic Chernov metric (\ref{rheon-B-M}) produces the
adapted metrical $d-$tensor%
\begin{equation}
\mathbb{G}_{[3]}=h_{11}dt\otimes dt+g_{[3]ij}dx^{i}\otimes
dx^{j}+h^{11}g_{[3]ij}\delta y_{1}^{i}\otimes \delta y_{1}^{j},
\label{gravit-pot-B-M}
\end{equation}%
where $g_{[3]ij}$ is given by (\ref{g-jos-(ij)}). This may be regarded as a
\textit{"non-isotropic gravitational potential"}. In such a physical
context, the nonlinear connection $\Gamma _{\lbrack 3]}$ (used in the
construction of the distinguished 1-forms $\delta y_{1}^{i}$) prescribes,
{most likely, a sort of} \textit{\textquotedblleft interaction\textquotedblright }
between $(t)$-, $(x)$- and $(y)$-fields.

We postulate that the non-isotropic gravitational potential $\mathbb{G}%
_{[3]} $ is governed by the \textit{geometrical Einstein equations}
\begin{equation}
\mbox{Ric }\left( C\Gamma _{\lbrack 3]}\right) -\frac{\mbox{Sc }\left(
C\Gamma _{\lbrack 3]}\right) }{2}\mathbb{G}_{[3]}\mathbb{=}\mathcal{KT},
\label{Einstein-eq-global}
\end{equation}%
where Ric $\left( C\Gamma _{\lbrack 3]}\right) $ is the \textit{Ricci
d-tensor} associated to the Cartan canonical connection $C\Gamma _{\lbrack
3]}$ (in Riemannian sense and using adapted bases), Sc $\left( C\Gamma
_{\lbrack 3]}\right) $ is the \textit{scalar curvature}, $\mathcal{K}$ is
the \textit{Einstein constant} and $\mathcal{T}$ is the intrinsic \textit{%
stress-energy} $d-$tensor of matter.

In this way, working with the adapted basis of vector fields (\ref{a-b-v}),
we can find the local geometrical Einstein equations for the rheonomic
Chernov metric (\ref{rheon-B-M}). Firstly, by direct computations, we find:

\begin{theorem}
The Ricci $d-$tensor of the Cartan canonical connection $C\Gamma _{\lbrack 3]}$
of the rheonomic Chernov metric (\ref{rheon-B-M}) has the following
effective local Ricci $d-$tensor {components}:%
   $$\begin{array}{ll}\medskip R_{ij}:=R_{ijr}^{r} & =\dfrac{1}{4}\varkappa _{11}^{1}\varkappa_{11}^{1}
        \mathbb{S}_{(i)(j)}^{(1)(1)}, \\
    P_{i(j)}^{\mbox{ }(1)}=P_{(i)j}^{(1)}:=P_{ij(r)}^{r\mbox{ }(1)} & =\dfrac{1}{2}
        \varkappa _{11}^{1}\mathbb{S}_{(i)(j)}^{(1)(1)},\medskip \\
    \mathbb{S}_{(i)(j)}^{(1)(1)}&
        =-9S_{[3]}^{pq1}S_{[3]}^{rm1}\left( S_{[3]ijp}S_{[3]qrm}-S_{[3]ipr}S_{[3]jqm}\right) +\medskip \\
    &\quad +\dfrac{1}{12}\dfrac{1}{S_{[3]111}}S_{[3]ij1}-\dfrac{1}{18}\dfrac{1}{S_{[3]111}^{2}}S_{[3]i11}S_{[3]j11},
    \end{array}$$
where $\mathbb{S}_{(i)(j)}^{(1)(1)}=S_{i(j)(r)}^{r(1)(1)}$ is the vertical Ricci $d-$tensor field.
\end{theorem}

\begin{proof}
Using the equality \eqref{C-(ijk)}, by laborious direct computations, we
obtain the following equalities ({we assume implicit summation by $r$ and $m$}):%
\begin{equation*}
{{\dfrac{\partial C_{i(j)}^{r(1)}}{\partial y_{1}^{r}}=3}}\frac{\partial
S_{[3]}^{rm1}}{\partial y_{1}^{r}}S_{[3]ijm}-\dfrac{1}{2}\dfrac{1}{S_{[3]111}%
}S_{[3]ij1}+\dfrac{5}{9}\dfrac{1}{S_{[3]111}^{2}}S_{[3]i11}S_{[3]j11},
\end{equation*}%
\begin{equation*}
{{\dfrac{\partial C_{i(r)}^{r(1)}}{\partial y_{1}^{j}}=3}}\frac{\partial
S_{[3]}^{rm1}}{\partial y_{1}^{j}}S_{[3]irm}-\dfrac{2}{3}\dfrac{1}{S_{[3]111}%
}S_{[3]ij1}+\dfrac{2}{3}\dfrac{1}{S_{[3]111}^{2}}S_{[3]i11}S_{[3]j11},
\end{equation*}%
\begin{eqnarray*}
{C_{i(j)}^{m(1)}C_{m(r)}^{r(1)}} &{=}%
&9S_{[3]}^{mp1}S_{[3]}^{rq1}S_{[3]ijp}S_{[3]mrq}- \\
&&-\dfrac{1}{2}\dfrac{1}{S_{[3]111}}S_{[3]}^{rq1}\left\{
S_{[3]irq}S_{[3]j11}+S_{[3]jrq}S_{[3]i11}\right\} - \\
&&-\dfrac{1}{6}\dfrac{1}{S_{[3]111}}S_{[3]ij1}+\dfrac{2}{9}\dfrac{1}{%
S_{[3]111}^{2}}S_{[3]i11}S_{[3]j11},
\end{eqnarray*}%
\begin{eqnarray*}
{C_{i(r)}^{m(1)}C_{m(j)}^{r(1)}} &{=}%
&9S_{[3]}^{mp1}S_{[3]}^{rq1}S_{[3]irp}S_{[3]mjq}- \\
&&-\dfrac{1}{2}\dfrac{1}{S_{[3]111}}S_{[3]}^{rq1}\left\{
S_{[3]irq}S_{[3]j11}+S_{[3]jrq}S_{[3]i11}\right\} - \\
&&-\dfrac{1}{12}\dfrac{1}{S_{[3]111}}S_{[3]ij1}+\dfrac{1}{6}\dfrac{1}{%
S_{[3]111}^{2}}S_{[3]i11}S_{[3]j11}.
\end{eqnarray*}%
Finally, taking into account that we have%
\begin{equation*}
\mathbb{S}_{(i)(j)}^{(1)(1)}=S_{i(j)(r)}^{r(1)(1)}={{\dfrac{\partial
C_{i(j)}^{r(1)}}{\partial y_{1}^{r}}}-{\dfrac{\partial C_{i(r)}^{r(1)}}{%
\partial y_{1}^{j}}}%
+C_{i(j)}^{m(1)}C_{m(r)}^{r(1)}-C_{i(r)}^{m(1)}C_{m(j)}^{r(1)},}
\end{equation*}%
and using the equalities%
\begin{equation*}
\begin{array}{l}
\dfrac{\partial S_{[3]}^{rm1}}{\partial y_{1}^{r}}%
S_{[3]ijm}=-6S_{[3]}^{mp1}S_{[3]}^{rq1}S_{[3]ijp}S_{[3]mrq},\medskip \\
\dfrac{\partial S_{[3]}^{rm1}}{\partial y_{1}^{j}}%
S_{[3]irm}=-6S_{[3]}^{mp1}S_{[3]}^{rq1}S_{[3]irp}S_{[3]jmq},%
\end{array}%
\end{equation*}%
we obtain the required result.
\end{proof}

\begin{remark}
The vertical Ricci $d-$tensor $\mathbb{S}_{(i)(j)}^{(1)(1)}$ has the following
property of symmetry: $\mathbb{S}_{(i)(j)}^{(1)(1)}=\mathbb{S}_{(j)(i)}^{(1)(1)}.$
\end{remark}

\begin{proposition}
The scalar curvature of the Cartan canonical connection $C\Gamma _{\lbrack 3]}$ of the rheonomic Chernov metric (\ref{rheon-B-M}) is given by
   $$\mbox{Sc }\left( C\Gamma _{\lbrack 3]}\right) =\frac{4h_{11}+\varkappa_{11}^{1}\varkappa _{11}^{1}}{4}\cdot
        \mathbb{S}^{11},\quad\mbox{ where }\quad\mathbb{S}^{11}=g_{[3]}^{pq}\mathbb{S}_{(p)(q)}^{(1)(1)}.$$
\end{proposition}

\begin{proof}
The general formula for the scalar curvature of a Cartan connection is ({cf. \cite{Ner}})%
\begin{equation*}
\mbox{Sc }\left( C\Gamma _{\lbrack 3]}\right)
=g_{[3]}^{pq}R_{pq}+h_{11}g_{[3]}^{pq}\mathbb{S}_{(p)(q)}^{(1)(1)}.
\end{equation*}
\end{proof}

Describing the global geometrical Einstein equations (\ref%
{Einstein-eq-global}) in the adapted basis of vector fields (\ref{a-b-v}),
{it is known} the following important geometrical and physical result ({cf. \cite{Ner}}):

\begin{theorem}
The local \textbf{geometrical Einstein equations} that govern the
non-isotropic gravitational potential (\ref{gravit-pot-B-M}) (produced by
the rheonomic Chernov metric (\ref{rheon-B-M})) are given by:
\begin{equation}\left\{\begin{array}{l}
\medskip \xi _{11}\mathbb{S}^{11}h_{11}=\mathcal{T}_{11} \\
\medskip \dfrac{\varkappa _{11}^{1}\varkappa _{11}^{1}}{4\mathcal{K}}\mathbb{%
S}_{(i)(j)}^{(1)(1)}+\xi _{11}\mathbb{S}^{11}g_{ij}=\mathcal{T}_{ij} \\
\dfrac{1}{\mathcal{K}}\mathbb{S}_{(i)(j)}^{(1)(1)}+\xi _{11}\mathbb{S}%
^{11}h^{11}g_{ij}=\mathcal{T}_{(i)(j)}^{(1)(1)}%
\end{array}%
\right.  \label{E-1}
\end{equation}%
\medskip
\begin{equation}
\left\{
\begin{array}{lll}
0=\mathcal{T}_{1i}, & 0=\mathcal{T}_{i1}, & 0=\mathcal{T}_{(i)1}^{(1)},%
\medskip \\
0=\mathcal{T}_{1(i)}^{\mbox{ }(1)}, & \dfrac{\varkappa _{11}^{1}}{2\mathcal{K%
}}\mathbb{S}_{(i)(j)}^{(1)(1)}=\mathcal{T}_{i(j)}^{\mbox{ }(1)}, & \dfrac{%
\varkappa _{11}^{1}}{2\mathcal{K}}\mathbb{S}_{(i)(j)}^{(1)(1)}=\mathcal{T}%
_{(i)j}^{(1)},%
\end{array}%
\right.  \label{E-2}
\end{equation}%
\medskip where
\begin{equation*}
\xi _{11}=-\frac{4h_{11}+\varkappa _{11}^{1}\varkappa _{11}^{1}}{8\mathcal{K}%
}.
\end{equation*}
\end{theorem}

\begin{remark}
The local geometrical Einstein equations (\ref{E-1}) and (\ref{E-2}) impose
as the stress-energy $d-$tensor of matter $\mathcal{T}$ to be symmetrical. In
other words, the stress-energy $d-$tensor of matter $\mathcal{T}$ must {satisfy}
the local symmetry conditions%
\begin{equation*}
\mathcal{T}_{AB}=\mathcal{T}_{BA},\quad \forall \mbox{ }A,B\in \left\{ 1,%
\mbox{ }i,\mbox{ }_{(i)}^{(1)}\right\} .
\end{equation*}
\end{remark}

\subsection{Geometrical electromagnetic theory}

\hspace{5mm}In the paper \cite{Ner}, using only a given Lagrangian
function $L$ on the 1-jet space $J^{1}(\mathbb{R},M^{4})$, a geometrical
theory for electromagnetism was also {constructed}. In {this background of jet
relativistic rheonomic Lagrange geometry, we work}
with an \textit{electromagnetic distinguished }$2$\textit{-form}%
   $$\mathbb{F}=F_{(i)j}^{(1)}\delta y_{1}^{i}\wedge dx^{j},$$
where
   $$F_{(i)j}^{(1)}=\frac{h^{11}}{2}\left[g_{jm}N_{(1)i}^{(m)}-g_{im}N_{(1)j}^{(m)}+
    \left(g_{ir}L_{jm}^{r}-g_{jr}L_{im}^{r}\right) y_{1}^{m}\right] ,$$
which is characterized by the following \textit{geometrical Maxwell equations} \cite{Ner}
   $$\aru{\begin{array}{ll}F_{(i)j/1}^{(1)} &=\dfrac{1}{2}\mathcal{A}_{\left\{ i,j\right\} }
        \left\{\overline{D}_{(i)1|j}^{(1)}-D_{(i)m}^{(1)}G_{j1}^{m}+d_{(i)(m)}^{(1)(1)}R_{(1)1j}^{(m)}-\right. \mm
    &\quad\left. -\left[ C_{j(m)}^{p(1)}R_{(1)1i}^{(m)}-G_{i1|j}^{p}\right]h^{11}g_{pq}y_{1}^{q}\right\} ,\mm
        \dsum_{\{i,j,k\}}F_{(i)j\;|\;k}^{(1)}&=-\dfrac{1}{8}\sum_{\{i,j,k\}}\dfrac{\partial^{3}L}{\partial
        y_{1}^{i}\partial y_{1}^{p}\partial y_{1}^{m}}\left[ {\dfrac{\delta N_{(1)j}^{(m)}}{\delta x^{k}}}-
        {\dfrac{\delta N_{(1)k}^{(m)}}{\delta x^{j}}}\right] y_{1}^{p},\mm
    \dsum_{\{i,j,k\}}F_{(i)j}^{(1)}\!\!\!\!\!\!\!\svcov{k}&=0,\end{array}}$$
where $\mathcal{A}_{\left\{ i,j\right\} }$ {denotes} an alternate sum, $\sum_{\{i,j,k\}}$ means a cyclic sum and we have
   $$\begin{array}{l}\aru{
    \overline{D}_{(i)1}^{(1)}=\dfrac{h^{11}}{2}\dfrac{\delta g_{ip}}{\delta t}y_{1}^{p},\quad
        D_{(i)j}^{(1)}=h^{11}g_{ip}\left[-N_{(1)j}^{(p)}+L_{jm}^{p}y_{1}^{m}\right] ,\mm
    d_{(i)(j)}^{(1)(1)}=h^{11}\left[ g_{ij}+g_{ip}C_{m(j)}^{p(1)}y_{1}^{m}\right],\quad
        \overline{D}_{(i)1|j}^{(1)}=\dfrac{\delta \overline{D}_{(i)1}^{(1)}}{\delta
    x^{j}}-\overline{D}_{(m)1}^{(1)}L_{ij}^{m},}\mm
   \qquad G_{i1|j}^{k}=\dfrac{\delta G_{i1}^{k}}{\delta x^{j}}+G_{i1}^{m}L_{mj}^{k}-G_{m1}^{k}L_{ij}^{m},\mm
   \aru{
    F_{(i)j/1}^{(1)}=\dfrac{\delta F_{(i)j}^{(1)}}{\delta t}+F_{(i)j}^{(1)}
        \varkappa _{11}^{1}-F_{(m)j}^{(1)}G_{i1}^{m}-F_{(i)m}^{(1)}G_{j1}^{m},\mm
    F_{(i)j\;|\;k}^{(1)}=\dfrac{\delta F_{(i)j}^{(1)}}{\delta x^{k}}-
        F_{(m)j}^{(1)}L_{ik}^{m}-F_{(i)m}^{(1)}L_{jk}^{m},\mm
    F_{(i)j}^{(1)}\!\!\vcov{k}=\dfrac{\partial F_{(i)j}^{(1)}}{\partial y_{1}^{k}}-
        F_{(m)j}^{(1)}C_{i(k)}^{m(1)}-F_{(i)m}^{(1)}C_{j(k)}^{m(1)}.}\end{array}$$
For {the rheonomic Chernov} metric (\ref{rheon-B-M}) we have $L=F_{[3]}^{2}$
and, consequently, we obtain the electromagnetic $2$-form
\begin{equation*}
\mathbb{F}:=\mathbb{F}_{[3]}=0.
\end{equation*}

In conclusion, {the locally-Minkowski rheonomic Chernov geometrical} electromagnetic theory is trivial. In
our opinion, this fact suggests that the metric \eqref{rheon-B-M} has rather gravitational connotations than electromagnetic ones in its $h-$flat ($x-$independent) version, which leads to the need of considering $x-$dependent conformal deformations of the structure (as, e.g., recently proposed by Garas'ko in \cite{gaa}).

\section{Conclusion}

\hspace{5mm}In recent physical and geometrical studies (\cite{As1,GP,Pa1,Pa2}), an important role is played by the
Finslerian metric
   \zx{F-2}F_{[2]}(t,y)=\sqrt{h^{11}(t)}\cdot \sqrt{y_{1}^{1}y_{1}^{2}+y_{1}^{1}y_{1}^{3}+
        y_{1}^{1}y_{1}^{4}+y_{1}^{2}y_{1}^{3}+y_{1}^{2}y_{1}^{4}+y_{1}^{3}y_{1}^{4}}\zc
which produces the fundamental metrical $d-$tensor
   $$g_{ij}:=g_{[2]ij}=\frac{h_{11}(t)}{2}\frac{\partial^{2}F_{[2]}^{2}}{\partial y_{1}^{i}\partial
    y_{1}^{j}}=\frac{1}{2}\left(1-\delta _{ij}\right) \Rightarrow g^{jk}:=g_{[2]}^{jk}=\frac{2}{3}
    \left( 1-3\delta ^{jk}\right) .$$
The Finslerian metric (\ref{F-2}) generates the jet canonical nonlinear connection
   $$\Gamma _{\lbrack 2]}=\left( M_{(1)1}^{(i)}=-\varkappa _{11}^{1}y_{1}^{i},
    \mbox{ }N_{(1)j}^{(i)}=-\frac{\varkappa _{11}^{1}}{2}\delta _{j}^{i}\right)$$
and the Cartan $\Gamma _{\lbrack 2]}$-linear connection
   $$C\Gamma _{\lbrack 2]}=\left( \varkappa _{11}^{1},\mbox{ }G_{j1}^{k}=0,\mbox{ }
    L_{jk}^{i}=0,\mbox{ }C_{j(k)}^{i(1)}=0\right) .$$
For the Cartan connection $C\Gamma _{\lbrack 2]}$ all torsion $d-$tensors vanish, except
   $$R_{(1)1j}^{(k)}=\dfrac{1}{2}\left[ \dfrac{d\varkappa _{11}^{1}}{dt}
    -\varkappa _{11}^{1}\varkappa _{11}^{1}\right] \delta _{j}^{k},$$
and all curvature $d-$tensors are zero. Consequently, all Ricci $d-$tensors vanish and the scalar curvature cancels.
The geometrical Einstein equations \eqref{Einstein-eq-global} produced by the Finslerian metric \eqref{F-2} become trivial, namely
   $$0=\mathcal{T}_{AB},\quad \forall \mbox{ }A,B\in \left\{ 1,\mbox{ }i,\mbox{ }_{(i)}^{(1)}\right\} .$$
At the same time, the electromagnetic $2$-form associated to the Finslerian metric (\ref{F-2}) has the trivial form
   $$\mathbb{F}:=\mathbb{F}_{[2]}=0.$$
In conclusion, both the metric-tensor based geometrical (gravitational and electromagnetic) theories are shown to be trivial for the case of the rheonomic locally Finslerian Chernov metric (\ref{F-2}). Hence, for developing a non-trivial $h-$model, one may need to consider other closely related alternatives offered by $h-$conformally-deformed models or by $x-$dependent rheonomic Finsler metrics of $m-$root type.

\noindent \textit{Author's addresses: }
{\small

\medskip \noindent \textbf{Vladimir Balan}\newline
University Politehnica of Bucharest, Faculty of Applied Sciences,\newline
Department of Mathematics-Informatics I,\newline
Splaiul Independen\c{t}ei 313, RO-060042 Bucharest, Romania.\newline
\textit{E-mail:} vladimir.balan@upb.ro\newline
\textit{Website:} http://www.mathem.pub.ro/dept/vbalan.htm

\medskip \noindent \textbf{Mircea Neagu}\newline
University Transilvania of Bra\c{s}ov, Faculty of Mathematics and
Informatics,\newline
Department of Algebra, Geometry and Differential Equations,\newline
B-dul Iuliu Maniu, Nr. 50, BV 500091, Bra\c{s}ov, Romania.\newline
\textit{E-mails:} mircea.neagu@unitbv.ro, mirceaneagu73@yahoo.com\newline
\textit{Website:} http://www.2collab.com/user:mirceaneagu
}
\end{document}